\documentclass[12pt]{amsart}
\usepackage[utf8]{inputenc}
\usepackage{amssymb,amsmath,amsthm,mathrsfs,url}
\usepackage{multirow, comment}
\usepackage{graphics}
\usepackage{float}

\usepackage{booktabs} 
\usepackage{dcolumn} 

\newtheorem{theorem}{Theorem}

\newtheorem{proposition}[theorem]{Proposition}

\newtheorem*{problem}{Problem}

\theoremstyle{remark}

\numberwithin{theorem}{section} \numberwithin{equation}{section}

\newcommand{\FF}{\mathbb{F}_p}

\newcommand{\Q}{\mathbb{Q}}

\newcommand{\Z}{\mathbb{Z}}

\newcommand{\ve}{\varepsilon}

\DeclareFontFamily{U}{wncy}{}
\DeclareFontShape{U}{wncy}{m}{n}{<->wncyr10}{}
\DeclareSymbolFont{mcy}{U}{wncy}{m}{n}
\DeclareMathSymbol{\Sh}{\mathord}{mcy}{"58}

\newtheorem*{theorem*}{Theorem}

\newcommand{\bigOe}[1]{O_\varepsilon\left(#1\right)}

\title{Murmurations of Mestre-Nagao sums}

\author{Zvonimir Bujanovi\'c}
\address{Department of Mathematics\\ 
	University of Zagreb\\
	Bijeni\v{c}ka cesta 30\\
	10000 Zagreb\\
	Croatia}
\email{zvonimir.bujanovic@math.hr}

\author{Matija Kazalicki}
\address{Department of Mathematics\\ 
	University of Zagreb\\
	Bijeni\v{c}ka cesta 30\\
	10000 Zagreb\\
	Croatia}
\email{matija.kazalicki@math.hr}

\author{Lukas Novak}
\address{Department of Mathematics\\ 
	University of Zagreb\\
	Bijeni\v{c}ka cesta 30\\
	10000 Zagreb\\
	Croatia}
\email{lukas.novak@math.hr}

\begin{document}

\maketitle

\section{Abstract}

This paper investigates the detection of the rank of elliptic curves with ranks 0 and 1, employing a heuristic known as the Mestre-Nagao sum

\[
S(B) = \frac{1}{\log{B}} \sum_{\substack{p<B \\ \textrm{good reduction}}} \frac{a_p(E)\log{p}}{p},
\]

\noindent where $a_p(E)$ is defined as $p + 1 - \#E(\mathbb{F}_p)$ for an elliptic curve $E/\mathbb{Q}$ with good reduction at prime $p$. This approach is inspired by the Birch and Swinnerton-Dyer conjecture.

Our observations reveal an oscillatory behavior in the sums, closely associated with the recently discovered phenomena of murmurations of elliptic curves \cite{Hee_Lee_Oliver_Arithmetic_Pozdnyakov}. Surprisingly, this suggests that in some cases, opting for a smaller value of $B$ yields a more accurate classification than choosing a larger one. For instance, when considering elliptic curves with conductors within the range of $[40\,000,45\,000]$, the rank classification based on $a_p$'s with $p < B = 3\,200$ produces better results compared to using $B = 50\,000$. This phenomenon finds partial explanation in the recent work of Zubrilina \cite{Zubrilina}.

\section{Introduction}

Let $E$ be an elliptic curve over $\Q$ with conductor $N$. Mordell's theorem states that the group of rational points $E(\Q)$ is a finitely generated abelian group, $E(\Q)_{\mathrm{tors}}\times \Z^r$, where $E(\Q)_{\mathrm{tors}}$ is the torsion subgroup and $r$ is its (algebraic) rank. While the torsion subgroups have been well understood following Mazur's work, determining the rank remains an enigma. Its possible values are unknown, a question initially posed by Poincaré  \cite{Poincare}. Moreover, there's no consensus on whether the rank is unbounded. Traditionally, it was thought to be unbounded until studies by Watkins and Park et al. \cite{Watkins14,Watkins15,Park_et_all} proposed, using heuristic models, that only a finite number of curves has a rank exceeding $21$. Elkies holds the current record with a rank of $28$.

Finding high-rank curves poses challenges, partly due to the computational complexity of determining an elliptic curve's rank. No universally applicable algorithm exists for this task, primarily because finding rational points on elliptic curves is complex. Descent algorithms, commonly used, often reduce to a basic point search on auxiliary curves. To tackle this, researchers employ rank heuristics inspired by the Birch and Swinnerton-Dyer conjecture. These heuristics help identify potential high-rank elliptic curves, easing the computational burden.

For each prime of good reduction $p$, we define $a_p=p+1-\#E(\FF)$. For $p|N$, we set $a_p=0,-1,$ or $1$  if, respectively,  $E$ has additive, split multiplicative or non-split multiplicative reduction at $p$. The $L$-function attached to $E/\Q$ is then defined as an Euler product
$$L_E(s)=\prod_{p|\Delta}\left(1-\frac{a_p}{p^s} \right)^{-1}\prod_{p \nmid \Delta}\left(1-\frac{a_p}{p^s}+\frac{p}{p^{2s}} \right)^{-1},$$
which converges absolutely for $\Re(s) > 3/2$ and extends to an entire function by the Modularity theorem \cite{Wiles,Breuil_Conrad_Diamond_Taylor}. The Birch and Swinnerton-Dyer (BSD) conjecture states that the order of vanishing of $L_E(s)$ at $s=1$ (the quantity known as analytic rank) is equal to the rank of $E(\Q)$.

Mestre \cite{Mestre82} and Nagao \cite{Nagao92}, and later others  \cite{Elkies_new_rank_records,Bober}, motivated by BSD conjecture, introduced certain sums (see Section 2 in \cite{Kazalicki_Vlah} for one list of sums) which are aimed at detecting curves with high analytic rank. In an abuse of terminology, we refer to all such sums as the Mestre-Nagao sums. 

In this paper we will study the following sum

\begin{align*}
	S(B)&=\frac{1}{\log{B}}\sum_{\substack{p<B, \\ \textrm{ good  reduction}}} \frac{a_p(E)\log{p}}{p}.\\
\end{align*}

The sum was thoroughly analyzed in \cite{Kim_Murty}, demonstrating that if the limit $\lim_{B\rightarrow \infty} S(B)$ exists, it converges to $-r_{an}+1/2$, where $r_{an}$ represents the analytic rank of $E/\Q$.

The classification of the rank of elliptic curves based on $a_p$ coefficients and conductor was explored in \cite{Kazalicki_Vlah} (see also \cite{Hee_Lee_Oliver_Arithmetic}), wherein the authors trained a deep convolutional neural network (CNN) for this purpose. As a benchmark for the CNN's performance, they also trained simple fully connected neural networks using the value of one of the six Mestre-Nagao sums (one of which was $S(B)$) at fixed $B=1\,000, 10\,000$ or $100\,000$ along with the conductor. Remarkably, the models encountered the most difficulty when classifying curves of rank zero and one, although this task is easily distinguishable for humans due to the Parity conjecture.

Therefore, our focus in this paper lies on the following classification problem.

\begin{problem} 
For a fixed $B > 0$, consider an elliptic curve $E/\mathbb{Q}$ with conductor $N$ and rank equal to $0$ or $1$. Our objective is to estimate the rank of $E$ based on the values of $S(B)$ and $N$.
\end{problem}

Essentially, given a specific $B$ and a conductor $N$ (or in practice, a conductor range such as $[N,N+10\sqrt{N}]$), the task is to identify the optimal cutoff value $C(N)$. This value distinguishes between curves with $S(B) > C(N)$, classified as rank $0$, and those with $S(B) \leq C(N)$, classified as rank $1$. Given the conjectural convergence of $S(B)$, one would anticipate that increasing $B$ would consistently enhance the classification quality. After all, having more coefficients $a_p$ should lead to a better approximation of the $L$-function. 

Surprisingly, this assumption appears not to hold true. We observe this in the Figure \ref{fig:rank01}, which shows the averages of $S(B)$ (along with $90\%$ confidence intervals) separately for curves of rank $0$ and $1$ in the conductor range $[40\,000,45\,000]$. In all our experiments, we are using Balakrishnan et al. \cite{BalakrishnanDatabase} database  of elliptic curves.

\begin{figure}
	\centering
	\resizebox{0.9\textwidth}{!}{\includegraphics{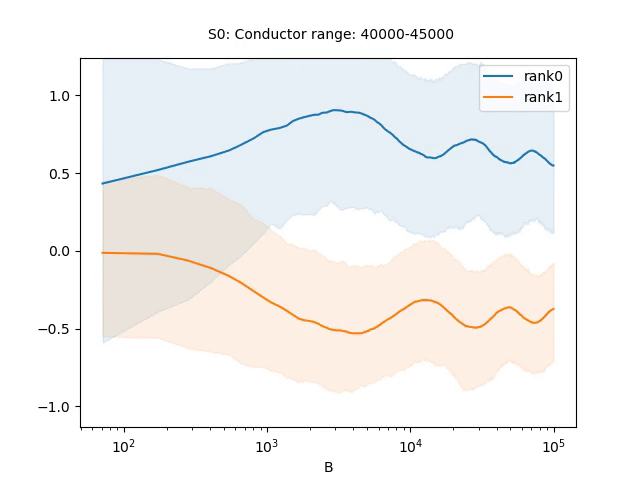}}
	\caption{Average values of $S(B)$ and their corresponding $90\%$ confidence intervals are computed for $1026$ curves of rank $0$ and $1485$ curves of $1$, within the conductor range $[40\,000, 45\,000]$.}
	\label{fig:rank01}
\end{figure}

The figure suggests that we can expect better classification quality if we choose $B$ corresponding to the first local maximum, which occurs at around $B = 0.08 \times 40\,000 = 3\,200$, rather than opting for a much larger $B$ value such as $B = 50\,000$.
Indeed, for the first choice of $B$, the optimal cutoff is $C =0.1368$, with which we can correctly classify $98.73\%$ of curves in the given conductor range, while for the second choice of $B$, the optimal cutoff is $C = 0.0694$, resulting in $97.85\%$ correct classification!

Interestingly, the locations of the first few local maxima can be fairly precisely predicted. In the conductor range $[N,N+10\sqrt{N}]$, the first local maximum is at around $B = 0.08N$, the second one is at $B = 0.65N$, and the third one is at $B = 1.7N$. For an illustrative demonstration of this phenomena, please refer to the animation \cite{Bujanovic_Kazalicki_Animation} that presents the supplementary material for this manuscript.

We can partially explain these findings by linking it to the recent work on the murmurations of elliptic curves. 

In their study, He, Lee, Oliver, and Pozdnyakov \cite{Hee_Lee_Oliver_Arithmetic_Pozdnyakov} discovered a striking oscillation pattern in the averages of the $a_p$ coefficients of elliptic curves with fixed rank (zero or one) and conductor within a specified range. See Figure \ref{fig:murmurations_curves} taken from their paper for an illustration.

\begin{figure}
	\centering
	\resizebox{0.6\textwidth}{!}{\includegraphics{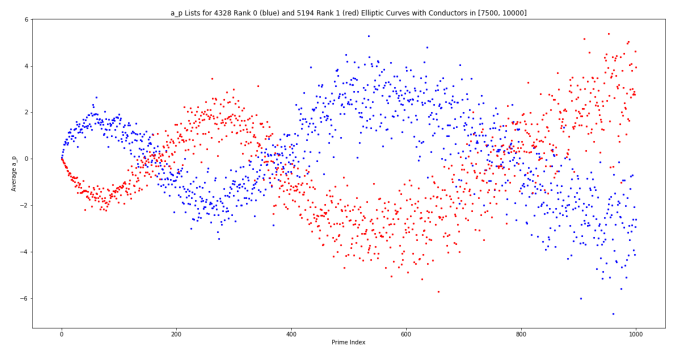}}
	\caption{Distribution of averages of $a_p$'s of elliptic curves of rank $0$ (blue) and $1$ (red) with conductors in $[7\,500,10\,000]$. The figure is sourced from \cite{Hee_Lee_Oliver_Arithmetic_Pozdnyakov}.}
	\label{fig:murmurations_curves}
\end{figure}

Subsequently, Sutherland observed similar phenomena in more general families of $L$-functions, including cusp forms with fixed root numbers. Zubrilina  \cite{Zubrilina} provided an explanation for these phenomena within the context of families of cusp forms. 

By applying the Eichler-Selberg formula to the composition of Hecke and Atkin-Lehner operators, Zubrilina derived an asymptotic formula for the average of $a_p(f) \epsilon(f)$, where $f$ ranges over newforms in the spaces of cusp forms $S_2(\Gamma_0(N))$ for $N$ squarefree in the interval $[X,X+Y]$. Here, $\epsilon(f)$ represents the root number, and $a_p(f)$ denotes the $p$-th Fourier coefficient of $f$. Please refer to Theorem 1 in \cite{Zubrilina} for the precise formulation of this result.

The main term of the average is a function of $y=\frac{p}{X}$ and is equal to
$$M(y)=C_1\sqrt{y}+C_2\sum_{1\le r\le 2\sqrt{y}} c(r)\sqrt{4y-r^2}-C_3 y,$$
where $C_i$ are explicit positive constants and $c(r)$ denotes an explicit positive function (see Section \ref{sec:max}), while the error term depends on $Y$.

We can use the formula for $M(y)$ as a heuristics, disregarding all the error terms, to approximate the averages of $S(B)$ over the families of rank $0$ and $1$, or more precisely, to approximate the average of $\frac{1}{\log{B}}\sum_{p<B} \frac{a_p(E)\epsilon(E)\log{p}}{p}$ for $E$ in a conductor range $[X,X+Y]$ with the expression $\frac{1}{\log{B}}\sum_{p<B} \frac{M(p/X)\log{p}}{p}$. For a fixed $X=N$, by setting $B=x N$, we can define the function 
$$f(x)=\frac{1}{\log{x N}}\sum_{p<x N} \frac{M(p/N)\log{p}}{p}.$$
Figure \ref{fig:f(x)} shows the graph of $f(x)$ for $N=100\,000$.

\begin{figure}
	\centering
	\resizebox{0.6\textwidth}{!}{\includegraphics{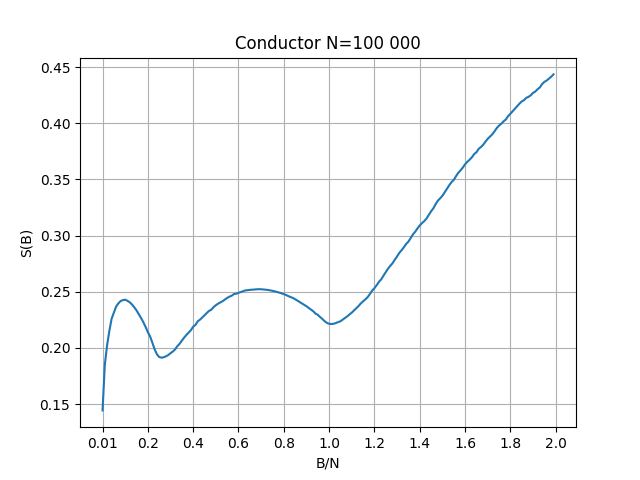}}
	\caption{Graph of $f(x)$ for $N=100,000$.}
	\label{fig:f(x)}
\end{figure}

The first two local maxima, located approximately at $0.11$ and $0.71$ respectively, closely correspond to our earlier observations from the data, which were approximately $0.08$ and $0.65$ respectively. However, the anticipated third maximum at $1.7$ is not discernible in the graph, likely due to the neglected error terms. For a detailed analysis of function $f(x)$, please refer to Section \ref{sec:max}. In particular, Proposition \ref{thm: loc_max_bound} implies that as $N$ tends to infinity, the first maximum of $f(x)$ converges to $\frac{C_1^2}{C_3^2} \approx 0.14261$, while the second converges to $\lambda\approx 0.75085$.

\section[Local maxima of $\textnormal{$f$}(x)$]{Local maxima of $\textnormal{$f$}(x)$} \label{sec:max}

In this section we will approximate the first two points of local maxima of the function
\[
f(x)=\frac{1}{\log{xN}}\sum_{p<xN}\frac{M(p/N)\log{p}}{p}.
\]
where $M_2(y)$ is the weight $2$ murmuration density defined as
\[
M_2(y)=C_1\sqrt{y}+C_2\sum_{1\leq r \leq 2\sqrt{y}}c(r)\sqrt{4y-r^2}-C_3y,
\]
where $C_1=D_2A$, $C_2=D_2B$, $C_3=D_2\pi$ are constants and
\begin{align*}
	c(r)&=\prod_{p\mid r}\left(1+\frac{p^2}{p^4-2p^2-p+1}\right), \\
	A&=\prod_{p}\left(1+\frac{p}{(p+1)^2(p-1)}\right), \\
	B&=\prod_{p} \frac{p^4-2p^2-p+1}{(p^2-1)^2},\\
	D_2&=\frac{12}{\pi \prod_{p}\left(1-\frac{1}{p^2+p}\right)}.
\end{align*}
\subsection[Estimations for  $\textnormal{$f$}(x)$]{Estimations for $\textnormal{$f$}(x)$}	

We start by observing that the sum $\sum_{1\leq r \leq 2\sqrt{p/N}} c(r)\sqrt{4p/N-r^2}$ in the function $M(p/N)$ vanishes for $0<x<1/4$ (i.e.\ for primes $p<N/4$) and for $1/4\leq x <1$ (i.e.\ for primes $N/4\leq p <N$) it only evaluates in $r=1$. Thus, we will observe the function $f(x)$ on the intervals $\left\langle0, \frac{1}{4}\right\rangle$ and $\left [\frac{1}{4}, 1 \right\rangle$.

Assume that $0<x<1/4$. By plugging in the formula for $M(p/N)$ we have that
\[
f(x)=\frac{1}{\log{xN}}\sum_{p\leq xN}\left(C \frac{\log{p}}{\sqrt{p}}-D\log{p}\right).
\]
where $C=\frac{C_1}{\sqrt{N}}$ and $D=\frac{C_3}{N}$.

Observe that $\sum_{p\leq xN} \log{p} = \vartheta(xN)$ where $\vartheta(x)$ is the first Chebyshev function.

Using Abel's summation formula we get
\[
\sum_{p\leq xN} \frac{\log{p}}{\sqrt{p}} = \frac{\vartheta(xN)}{\sqrt{xN}} + \int_{2}^{xN}\frac{\vartheta(t)}{2t^{\frac{3}{2}}}\, dt.
\]

By assuming the Riemann hypothesis (RH), we have that $\vartheta(x)=x + \bigOe{x^{\frac{1}{2}+\ve}}$. Thus, we get the following estimates:    
\begin{align*}
	\sum_{p\leq xN}\log{p} &=xN+\bigOe{(xN)^{\frac{1}{2}+\ve}},\\   
	\sum_{p\leq xN} \frac{\log{p}}{\sqrt{p}} 
	&= \sqrt{xN}+\bigOe{(xN)^\ve} + \int_{2}^{xN}\frac{1}{2\sqrt{t}}\, dt + \bigOe{\int_{2}^{xN}\frac{1}{2t^{1-\ve}}\, dt} \\
	&=2\sqrt{xN}-\sqrt{2}+\bigOe{(xN)^\ve}.
\end{align*}

From here we finally have 
\begin{align*}
	f(x)
	&=\frac{1}{\log{xN}}\left(2C_1\sqrt{x}-C_3x-\frac{C_1\sqrt{2}}{\sqrt{N}}\right) + \bigOe{x^\ve N^{\ve-\frac{1}{2}}}+\bigOe{x^{\frac{1}{2}+\ve}N^{\ve-\frac{1}{2}}} \\
	&=\frac{1}{\log{xN}}\left(2C_1\sqrt{x}-C_3x-\frac{C_1\sqrt{2}}{\sqrt{N}}\right) + \bigOe{x^\ve N^{\ve-\frac{1}{2}}}.
\end{align*}
Note that in the last equation we used $x^{\frac{1}{2}+\ve} N^{\ve-\frac{1}{2}}\leq x^\ve N^{\ve-\frac{1}{2}}$ since $x<1$.

For $1/4\leq x <1$ we have that 
\[
f(x) = \frac{1}{\log{xN}}\sum_{p \leq xN}\left(C\frac{\log{p}}{\sqrt{p}}-D\log{p} \right) + \frac{E}{\log{xN}}\sum_{N/4<p\leq xN}\frac{\sqrt{4p-N}}{p}\log{p}.
\]
where $C=\frac{C_1}{\sqrt{N}}$, $D=\frac{C_3}{N}$ and $E=\frac{C_2}{\sqrt{N}}$.

Again by using Abel's summation formula and $\vartheta(x)=x+\bigOe{x^{\frac{1}{2}+\ve}}$ we get
\begin{align*}   
	&\sum_{N/4<p\leq xN}\frac{\sqrt{4p-N}}{p}\log{p} = \frac{\sqrt{4xN-N}}{xN}\vartheta(xN) - \int_{N/4}^{xN}\vartheta(t)\frac{N-2t}{t^2\sqrt{4t-N}}\, dt\\
	&= \sqrt{N(4x-1)}+\bigOe{(xN)^\ve}-\int_{N/4}^{xN}\frac{N-2t}{t\sqrt{4t-N}}\, dt + \bigOe{\int_{N/4}^{xN}\frac{N-2t}{t^{3/2-\ve}\sqrt{4t-N}}}\, dt \\
	&=2\sqrt{N(4x-1)} - 2\sqrt{N}\arctan{(\sqrt{4x-1})} + \bigOe{(xN)^\ve}.
\end{align*}

Finally, by plugging this back into the above equation for $f(x)$ and using the previous estimates for the sums $\sum_{p\leq xN} \log{p}$ and $\sum_{p\leq xN} \frac{\log{p}}{\sqrt{p}}$, we obtain
\begin{align*}
	f(x) 
	&= \frac{1}{\log{xN}}\left(2C_1\sqrt{x}+2C_2\sqrt{4x-1}-2C_2\arctan{(\sqrt{4x-1})}-C_3x-\frac{C_1\sqrt{2}}{\sqrt{N}}\right) \\
	&+ \bigOe{x^\ve N^{\ve-\frac{1}{2}}}.
\end{align*}

We summarize this results in the next proposition.
\begin{proposition}\label{thm: f(x)_est}
	With the notation as above, the following holds (under RH):
	\begin{enumerate}
		\item[i)] If $0<x<1/4$, then
		\[
		f(x)=\frac{1}{\log{xN}}\left(2C_1\sqrt{x}-C_3x-\frac{C_1\sqrt{2}}{\sqrt{N}}\right) + \bigOe{x^\ve N^{\ve-\frac{1}{2}}}.
		\]
		\item[ii)] If $1/4\leq x<1$, then
		\begin{align*}
			f(x) 
			&= \frac{1}{\log{xN}}\left(2C_1\sqrt{x}+2C_2\sqrt{4x-1}-2C_2\arctan{(\sqrt{4x-1})}-C_3x-\frac{C_1\sqrt{2}}{\sqrt{N}}\right)\\
			&+ \bigOe{x^\ve N^{\ve-\frac{1}{2}}}.
		\end{align*}
	\end{enumerate}
\end{proposition}

\subsection{Estimations of the local maxima}

By disregarding the error terms and calculating the local maxima of the main terms obtain for $f(x)$ in Proposition~\ref{thm: f(x)_est} we get good estimates for the first and the second local maxima of $f(x)$. This estimates are shown in the Table~\ref{tab:loc_max} below.

\begin{table}[htbp]
	\centering
	\caption{Estimates for the local maxima of $f(x)$}
	\label{tab:loc_max}
	\begin{tabular}{ c D{.}{.}{5.5} D{.}{.}{5.5} } 
		\toprule
		{$N$} & \multicolumn{1}{c}{First local maxima} & \multicolumn{1}{c}{Second local maxima} \\ 
		\midrule
		{$10^4$} & 0.10023 & 0.69381 \\ 
		{$10^5$} & 0.11077 & 0.70510 \\ 
		{$10^6$} & 0.11724 & 0.71294 \\ 
		{$10^7$} & 0.12156 & 0.71856 \\ 
		{$10^8$} & 0.12334 & 0.72276 \\ 
		\bottomrule
	\end{tabular}
\end{table}

Although we are not able to give explicit formulas for the local maxima of the main terms, we have an upper bound for them and we can describe their limit as $N \to \infty$.

\begin{proposition}\label{thm: loc_max_bound}
Let $x_1(N)$ denote the local maximum of the main term in Proposition~\ref{thm: f(x)_est} i), and let $x_2(N)$ denote the local maximum of the main term in Proposition~\ref{thm: f(x)_est} ii). Additionally, let $m_1(N)$ and $m_2(N)$ represent the first and second local maxima of $f(x)$, respectively. With the above notation, the following statement holds
	\begin{enumerate}
		\item[i)] 
		$m_1(N)\leq \dfrac{A^2}{\pi^2}=0.14261\dotso$.
		
		\item[ii)]
		$m_2(N) \leq \dfrac{A^2+4B^2+\sqrt{(A^2+4B^2)^2-2\pi^2B^2}}{\pi^2}=0.76881\dotso$.
		
		\item[iii)]
		$\lim_{N \to \infty} m_1(N)=\lim_{N \to \infty} x_1(N)=\dfrac{A^2}{\pi^2}$.
		
		\item[iv)]
		$\lim_{N \to \infty} m_2(N)=\lim_{N \to \infty} x_2(N)=\lambda$ where $\lambda \geq 1/4$ satisfies the equation
		\[
		A\sqrt{(4\lambda-1)\lambda}+4B\lambda-\pi\lambda\sqrt{4\lambda-1}=B.    
		\]
	\end{enumerate}
\end{proposition}
\begin{proof}
	\mbox{}\\
	\begin{enumerate} 
		\item[i)]
		The inequality follows directly by observing when the main term of the sum in $f(x)$ (for $0<x<1/4$) changes the sign from plus to minus. This happens exactly at $x=\frac{A^2}{\pi^2}$. Hence, $m_1(N)\leq \frac{A^2}{\pi^2}$.
		
		\item[ii)]
		Similar as in i) we observe where the main term of the sum in $f(x)$ (for ${1/4\leq x<1}$) changes the sign. This yields the inequality 
		\[
		A\sqrt{x}-\pi x + B\sqrt{4x-1}<0.
		\]
		Since $A\sqrt{x}+B\sqrt{4x-1}\leq \sqrt{2(A^2+4B^2)x-2B^2}$ by AM-QM inequality we can instead look at the weaker inequality 
		\[
		\sqrt{2(A^2+4B^2)x-2B^2}-\pi x<0.
		\]
		The above inequality leads to a quadratic inequality with solution 
		\[
		x \in \left\langle\dfrac{A^2+4B^2-\sqrt{(A^2+4B^2)^2-2\pi^2B^2}}{\pi^2} , \dfrac{A^2+4B^2+\sqrt{(A^2+4B^2)^2-2\pi^2B^2}}{\pi^2}\right\rangle.
		\]
		Hence, $m_2(N) \leq \dfrac{A^2+4B^2+\sqrt{(A^2+4B^2)^2-2\pi^2B^2}}{\pi^2}$.
	\end{enumerate}
	
	\item[iii)]
	As $x_1(N)$ is the local maxima of the function $\frac{1}{\log{xN}}\left(2C_1\sqrt{x}-C_3x-\frac{C_1\sqrt{2}}{\sqrt{N}}\right)$, after calculating the derivative we get that $x_1(N)$ satisfies
	\[
	A\sqrt{x_1(N)}-\pi x_1(N)=\frac{2A\sqrt{x_1(N)}-\pi x_1(N)-\frac{A\sqrt{2}}{\sqrt{N}}}{\log{N}+\log{x_1(N)}}.
	\]
	Since we are only interested in the local maxima, instead of analyzing the function $\frac{1}{\log{xN}}\left(2C_1\sqrt{x}-C_3x-\frac{C_1\sqrt{2}}{\sqrt{N}}\right)$ on $\langle 0, 1/4\rangle$ we can analyze it on $[\delta, 1/4\rangle$ for a fixed small number $\delta>0$. This gives us the bounds $\delta\leq x_1(N)<1/4$. By using these bounds in the above equation, it follows that the limit $\lim_{N \to \infty} x_1(N)$ exists. Denote $\alpha=\lim_{N \to \infty} x_1(N)$. 
	
	Finally, by letting $N \to \infty$ in the above equation we get that the right-hand side converges to $0$ and therefore we have $A\sqrt{\alpha}-\pi \alpha=0$.
	From here we have that $\alpha = \frac{A^2}{\pi^2}$ (since $\alpha>0$).
	
	Note that $\lim_{N \to \infty} m_1(N)=\lim_{N \to \infty} x_1(N)$ follows directly from Proposition~\ref{thm: f(x)_est} since the error term goes to $0$.
	
	\item[iv)]
	Similar as in iii) we get that the local maxima $x_2(N)$ satisfies the equation
	\begin{align*}
		&A\sqrt{x_2(N)}+\frac{4Bx_2(N)}{\sqrt{4x_2(N)-1}}-\frac{B}{\sqrt{4x_2(N)-1}} - \pi x_2(N)=\\
		&=\frac{\pi x_2(N)+\frac{A\sqrt{2}}{\sqrt{N}}-2A\sqrt{x_2(N)}-2B\sqrt{4x_2(N)-1}+2B\arctan{\sqrt{4x_2(N)-1}}}{\log{N}+\log{x_2(N)}}.
	\end{align*}
	Since $1/4\leq x_2(N)<1$, from the above equation follows that the limit $\lim_{N \to \infty}x_2(N)$ exists. Denote $\lambda=\lim_{N \to \infty}x_2(N)$.
	
	By proceeding the same way as in iii) we get $\lim_{N \to \infty} m_2(N)=\lim_{N \to \infty} x_2(N)$.
	
	By letting $N \to \infty$ in the above equation we get that the right-hand side converges to $0$ and this implies that $\lambda$ satisfies
	\[
	A\sqrt{(4\lambda-1)\lambda}+4B\lambda-\pi\lambda\sqrt{4\lambda-1}=B.    
	\]
\end{proof}

\section{Future work}

The analysis of the main term of the averages of $a_p(f)\epsilon(f)$ from \cite{Zubrilina} offered a qualitative heuristic explanation for the first two observed local maxima in our data on elliptic curves. It would be interesting to investigate whether the analysis of the error terms could shed light on the presence of the third local maximum. Moreover, exploring this phenomenon in the context of cusp forms could lead to formulating and proving precise theorems.

\section*{Acknowledgments}
This work was supported by the project “Implementation of cutting-edge research and its application as part of the Scientific Center of Excellence for Quantum and Complex Systems, and Representations of Lie Algebras“, PK.1.1.02, European Union, European Regional Development Fund.
The second and third authors were supported by the Croatian Science Foundation under the project no.~IP-2022-10-5008.

\bibliographystyle{alpha}
\bibliography{bibliography}
\end{document}